\documentclass[letterpaper,conference]{ieeeconf}

\IEEEoverridecommandlockouts                       
\usepackage{cite}
\usepackage{lmodern}
\usepackage{amsmath}
\usepackage{graphicx}
\usepackage{bbm}
\usepackage[colorlinks=true, allcolors=black]{hyperref}
\usepackage{amssymb}
\usepackage{amsfonts}
\usepackage{colortbl}
\usepackage{dsfont}
\usepackage{bm}
\usepackage{algorithm}
\usepackage{algorithmic}
\usepackage{dsfont}
\usepackage{color}
\usepackage{tabularx}
\usepackage{lipsum}
\usepackage{mathtools}
\usepackage{listings}

\usepackage[usenames,dvipsnames]{pstricks}

\newcommand{\espE}{\mathbb{E} }
\newtheorem{theorem}{Theorem}
\newtheorem{lemma}[theorem]{Lemma}
\newtheorem{proposition}[theorem]{Proposition}

\newtheorem{remark}[theorem]{Remark}

\def\BibTeX{{\rm B\kern-.05em{\sc i\kern-.025em b}\kern-.08e=
    T\kern-.1667em\lower.7ex\hbox{E}\kern-.125emX}}

\begin{document}

\title{Optimal Control of an Interconnected SDE - Parabolic PDE System}

\author{Gabriel Velho$^{1}$, 
\thanks{$^{1}$Universit\'e Paris-Saclay, CNRS, CentraleSup\'elec, Laboratoire des signaux et syst\`emes, 91190, Gif-sur-Yvette, France. This project has received funding from the Agence Nationale de la Recherche (ANR) via grant PANOPLY ANR-23-CE48-0001-01.} 
Jean Auriol$^{1}$, 
Islam Boussaada$^{2}$, 
Riccardo Bonalli$^{1}$ 
\thanks{$^{2}$Universit\'e Paris-Saclay, CNRS, CentraleSup\'elec, Inria, Laboratoire des signaux et syst\`emes, 91190, Gif-sur-Yvette, France  \& IPSA, 94200, Ivry sur Seine, France.}%
}

\maketitle
\thispagestyle{empty}

\begin{abstract}
In this paper, we design a controller for an interconnected system where a linear Stochastic Differential Equation (SDE) is actuated through a linear parabolic heat equation.
These dynamics arise in various applications, such as coupled heat transfer systems and chemical reaction processes that are subject to disturbances.
Our goal is to develop a computational method for approximating the controller that minimizes a quadratic cost associated with the state of the SDE component.
To achieve this, we first perform a change of variables to shift the actuation inside the PDE domain and reformulate the system as a linear Stochastic Partial Differential Equation (SPDE). We use a spectral approximation of the Laplacian operator to discretize the coupled dynamics into a finite-dimensional SDE and compute the optimal control for this approximated system. The resulting control serves as an approximation of the optimal control for the original system.
We then establish the convergence of the approximated optimal control and the corresponding closed-loop dynamics to their infinite-dimensional counterparts. Numerical simulations are provided to illustrate the effectiveness of our approach.
\end{abstract}

\maketitle


\section{Introduction}\label{IN}

Interconnected systems involving parabolic Partial Differential Equations (PDEs) and Ordinary Differential Equations (ODEs) have recently attracted significant interest. In particular, systems where the coupling occurs at one boundary of the PDE and the control input is applied at the other have been a focus of recent research. 
These dynamics arise in various applications, such as coupled heat transfer systems and chemical reaction processes \cite{mohammadi_optimal_2015}.

To control these coupled systems, stabilization techniques from parabolic PDE control have been developed. A widely used approach is finite-dimensional approximation \cite{ito_strong_1987, banks_approximation_1994, morris_controller_2020}, which discretizes the system to make control design more tractable. Another effective method is backstepping \cite{krstic_boundary_2008}, which systematically transforms the system into a stable target form. These techniques have been successfully applied to parabolic PDE-ODE systems, as seen in \cite{tang_state_2011, antonio_susto_control_2010, wang_output_2019} for backstepping and \cite{mohammadi_optimal_2015} for the dimension reduction approach.


In realistic settings, dynamical systems are often subject to disturbances arising from measurement noise, parameter uncertainties, or external influences \cite{bonalli_sequential_2022,lew_risk_2023}. These disturbances can significantly affect the system behavior, making their mitigation essential for ensuring reliability and safety in control applications. Stochastic Differential Equations (SDEs) provide a powerful and flexible framework for modeling such uncertainties \cite{touzi_optimal_2013,bonalli_first_2023}. By leveraging stochastic control techniques, one can design robust stabilizing controllers that effectively mitigate random fluctuations. In particular, optimal stochastic control plays a key role in minimizing variance and improving system reliability \cite{lew_sample_2024}.

In interconnected PDE-ODE systems, accounting for uncertainties in the ODE dynamics is often crucial for achieving both efficiency and robustness. Several methods have been proposed for hyperbolic systems, addressing bounded or model-based noise in the ODE \cite{redaud_output_2024, deutscher_backstepping_2017}, as well as for stochastic differential equations (SDEs) 
\cite{velho_stabilization_2024 ,velho_stabilization_2025}. 
Recently, significant progress has been made in the control of parabolic Stochastic Partial differential Equations (SPDE). In particular, the spectral approximation method has proven effective in stabilizing semilinear systems \cite{wang_predictor_2023, wang_constructive_2023}.
However, to the best of our knowledge, no existing approach specifically addresses noise mitigation in the context of coupled parabolic PDE-SDEs. 
This perspective enables the modeling of a broader class of dynamical systems, including processes governed by diffusion, such as heat transfer phenomena, chemical reactions, or the spread of biological populations.
Furthermore, leveraging the structural properties of parabolic PDEs can lead to more effective control strategies for minimizing stochastic disturbances. For instance, the finite-dimensional approximation technique, which is well-suited for parabolic systems, may not be directly applicable to hyperbolic systems \cite{auriol_late-lumping_2019}.

This study aims to provide the first systematic approach to controlling interconnected parabolic PDE-SDE systems. Specifically, we bridge the gap between deterministic PDE control and stochastic control by integrating methodologies from both topics.
Our contributions are twofold:
\begin{enumerate}
    \item Reformulating the optimal control problem for an interconnected parabolic PDE-linear SDE system as a well-posed Linear Quadratic (LQ) control problem within the SPDE framework.
    \item Extending the finite-dimensional approximation methods used to compute optimal LQ controls for parabolic PDEs \cite{morris_controller_2020, mohammadi_optimal_2015} to the stochastic setting. Specifically, we establish the convergence of the approximated Riccati operator and the resulting controlled dynamics.
\end{enumerate}



The remaining of the paper is organized as follows: After recalling some standard notations the problem under consideration is stated in Section \ref{PF}.  In Section \ref{SPDE} it is shown that the considered coupled PDE+SDE can be recast into an SPDE. We give our control strategy in Section \ref{DS} and prove convergence of our control method in Section \ref{CVR}. Finally, some numerical results are shown in Section \ref{SIM}.


\section{Problem Formulation}\label{PF}

\subsection{Notations}

We assume that we are given a filtered probability space $(\Omega, \mathcal{F} \triangleq (\mathcal{F}_t)_{t \in [0,\infty)}, \mathbb{P})$, 
and that stochastic perturbations are due to a one-dimensional Wiener process $W_t$, which is adapted to filtration $\mathcal{F}$.
Let $T > 0$ be some given time horizon. For a Hilbert space $H$, we denote by 
$L_\mathcal{F}^2([0,T] ; H)$ the set of square integrable processes $Z: [0,T]\times\Omega \to H$ that are $\mathcal{F}$--progressively measurable, whereas the subset $C^2_\mathcal{F}([0,T] ; H) \subseteq L_\mathcal{F}^2([0,T] ; H)$ contains processes whose sample paths are continuous. The associated norm is
$$
\Vert Z(\cdot) \Vert_{C^2_\mathcal{F}([0,T] ; H)} =  \espE \left[  \underset{0 \leq t \leq T}{\sup}  \Vert Z(t) \Vert_H^2  \right].
$$
We also recall the \textit{Burkholder-Davis-Gundy} (B-D-G) inequality for stochastic integrals \cite{marinelli_maximal_2016}: for $p\geq 1$,
\begin{equation}\label{eq:Burkhloder_Davis_Gundy_ineq_stoch_int}
    \hspace{-0.2cm}\espE \left[  \underset{0 \leq t \leq T}{\sup}  \Vert Z(t) \Vert_H^p  \right] \leq C_p \espE \left[ \left( \int_0^T  \Vert Z(t) \Vert_H^p dt \right)^{p/2} \right],
\end{equation}
where $C_p$ is a constant which depends on $p$ uniquely. The spaces of semi-definite and definite positive symmetric matrices in $\mathbb{R}^n$ are denoted by $\mathcal{S}^+_n$ and $\mathcal{S}^{++}_n$, respectively. 
We denote by $H^1(0,1) \triangleq \bigl \{u \in L^2([0,1] ; \mathbb{R}), \ \text{such that} \  u' \in L^2([0,1] ; \mathbb{R})  \bigr \}$. It is a Hilbert space with the scalar product
$$
< u , v >_{H^1} \triangleq \int_0^1 u(x) v(x) dx + \int_0^1 u'(x) v'(x) dx.
$$

\subsection{Control system}

We consider coupled PDE+SDE systems of the form
\begin{equation}\label{eq:systeme_original_coupled}
\left\{ \begin{array}{l} u_t(t,x) = \Delta u(t,x)  +  c u(t,x) \\ 
dX_t = \bigl( A X_t + B u(t,0) \bigr) dt + \bigl( C X_t + D u(t,0) \bigr) dW_t \\
u_x(t,1) = U(t) , \quad u_x(t,0) = 0 \\
u(0,x) = u_0(x), \quad X(0) = X_0
\end{array}  \right.
\end{equation}
in the time-space domain $[0, T] \times [0,1]$, where $T>0$. The state of the system is $(X_t, u(t,\cdot)) \in \mathbb{R}^d \times H^1([0,1])$. The coefficient $c$ is a constant modeling the rate of growth ($c>0$) or decay ($c<0$) of the solution $u$ at each point in space. 
The matrices $A \in \mathbb{R}^{q\times q}, B \in \mathbb{R}^{q\times 1}, C \in \mathbb{R}^{q\times q}, D \in \mathbb{R}^{q\times 1}$ are also constant. The control input $U$ takes values in $\mathbb{R}$. 
We denote by $\text{Dom}(\Delta) \triangleq \bigl \{   u \in H^2(0,1), \ u_x(0) = u_x(1) = 0   \bigr \}$.
Finally, the initial condition $(X_0, u_0) \in \mathbb{R}^d \times \text{Dom}(\Delta)$ is assumed to be deterministic.  
Note that the terms $c u(t,x)$ and $A X_t$ can introduce instabilities in the PDE and the SDE if they are positive. 
%
The class of systems described by~\eqref{eq:systeme_original_coupled} naturally arises in scenarios where a heat equation governs one part of the system—such as a pipe transmitting heat—while it is coupled to a temperature system influenced by random perturbations. A typical example is the temperature regulation of a building, where external factors like fluctuating outdoor temperatures or varying sunlight exposure introduce stochastic disturbances. In these settings, the Neumann-type conditions model a set flow between the systems.  
We defer the proof of the well-posedness of system~\eqref{eq:systeme_original_coupled} to Section~\ref{SPDE}.

\subsection{Optimal Control Problem}
Our primary objective is to develop a method for computing a control that ensures the reliable regulation of our system. A natural approach is to design controls that minimize the variance of the SDE state \cite{lew_sample_2024}. To achieve this, we consider a quadratic cost of the form
\begin{equation}\label{eq:original_LQ_cost_constant}
J(U) \triangleq \espE \left[ \int_0^T \bigl( X^T_t Q X_t + r U(t)^2 \bigr) dt + X_T^T G X_T \right]  
\end{equation}
where $X_t$ is subject to the dynamic \eqref{eq:systeme_original_coupled}. The matrices $Q$ and $R$ are in $\mathcal{S}_q^+$, and $r>0$ is a positive constant. 
To obtain a controller that is both computationally tractable and practical for implementation, we minimize the cost over the set of feedback controls \cite{touzi_optimal_2013}.
To do so, we first perform a change of variables in the PDE, which allows the control to act within the domain rather than at the boundary. 
We then formulate the resulting dynamics within the SPDE framework, using the existing literature to ensure the existence of an optimal control. Finally, we approximate the optimal control using spectral methods and establish the convergence of both the control approximation and the associated cost toward their respective optimal values.

\begin{remark}
    Note that the cost function~\eqref{eq:original_LQ_cost_constant} does not explicitly account for the PDE state $u(t,x)$, as our primary focus is on controlling the SDE, with the PDE modeling the medium through which the control goes through. However, it would be possible (at the cost of additional notational complexities) to incorporate an additional term of the form $\espE \left[ \int_0^T < M u(t,\cdot), u(t,\cdot)>_{H^1} dt \right]$ in~\eqref{eq:original_LQ_cost_constant}, where $M$ is a positive self-adjoint operator. We leave it as a future work. 
\end{remark}


\section{Recasting into an SPDE LQ problem}\label{SPDE}

\subsection{Change of variable}

In this section, we perform a change of variable in the PDE state as in \cite{katz_finite-dimensional_2021}, which results in the input operator becoming bounded, at the cost of adding regularity on the control input $U$. 
Consider the new PDE state $z(t,x)$ defined by
\begin{equation}\label{eq:change_of_variable_PDE}
    z(t,x) \triangleq u(t,x) - \rho(x) U(t),
\end{equation}
where $\rho \in H^1([0,1])$ is the unique solution to the elliptic equation
\begin{equation}\label{eq:elliptic_PDE_rho}
 \left\{ \begin{array}{l} 
    - \Delta \rho(x) + \bigl( \mu - c \bigr) \rho(x) = 0 \\
    \rho_x(0) = 0, \quad \rho_x(1) = 1
\end{array}  \right.
\end{equation}
where $\mu> c$. The parameter $\mu$ is required to guarantee the well-posedness of equation~\eqref{eq:elliptic_PDE_rho} through the Lax-Milgram theorem \cite[Corollary 5.8]{brezis_functional_2011}. In practice, we want to choose $\mu$ relatively close to $c$, as the change of variable \eqref{eq:dynamic_change_variable_control_U} in the control is proportional to a term $e^{\mu t}$. 
The variable $z$ verifies the PDE system
$$
dz(t,x) = \Delta z(t,x) dt + c z(t,x) + \mu \rho(x) U(t) dt - \rho(x) dU(t).
$$
Therefore, by choosing $U(t)$ such that
\begin{equation}\label{eq:dynamic_change_variable_control_U}
\left\{ \begin{array}{l}
    dU(t) = \mu U(t) dt + V(t) dt \\
    U(0) = U_0,
\end{array} \right.
\end{equation}
with $V(t)$ being the new control variable, we obtain the following dynamic on $z$ and $X_t$
\begin{equation}\label{eq:PDE_z_after_control_change_variable_V}
\left\{ \begin{array}{l} z_t(t,x) = \Delta z(t,x)  +  c z(t,x) - \rho(x) V(t) \\ 
dX_t = \bigl( A X_t + B ( z(t,0) - \rho(0) U(t) ) \bigr) dt \\
\hspace{4em} + \bigl( C X_t + D  ( z(t,0) - \rho(0) U(t) ) \bigr) dW_t \\
dU(t) = \mu U(t) dt + V(t) dt \\
z_x(t,1) = 0 , \quad z_x(t,0) = 0, \quad U(0) = U_0 \\
z(0,x) = u_0(x) - \rho(x) U_0, \quad X(0) = X_0.
\end{array}  \right.
\end{equation}
Note that we can retrieve $U$ from the new control $V$ with the formula $U(t) = e^{\mu t} U_0 + \int_0^t e^{\mu (t-s) } V(s) ds$. We keep the parameter $U_0$ undefined for the moment, as we use it to further minimize the cost function~\eqref{eq:original_LQ_cost_constant} later on.

\subsection{Well-posedness of the system}
We now prove the well-posedness of the coupled PDE-SDE system \eqref{eq:PDE_z_after_control_change_variable_V} and state some regularity properties of the solution.

\begin{proposition}\label{prop:well_posedness_coupled_system}
Let $V(\cdot)$ be a stochastic process in $L^2_{\mathcal{F}_t}([0,T], \mathbb{R})$. Then, equation \eqref{eq:PDE_z_after_control_change_variable_V} has a unique solution. Moreover, the state $( z , X)$ enjoy the following regularity:
\small
\begin{align*}
    & \forall \omega \in \Omega, \quad z(\omega)  \in C( [0,T] ; H^1(0,1) ) \ \cap \ L^2( [0,T] ; H^2(0,1) ), \\
    & X_{(\cdot)} \in C^2_{\mathcal{F}}([0,T] ; \mathbb{R}^d)
\end{align*}
\normalsize
\end{proposition}
\begin{proof}
We first prove the statement for the PDE state $z$ since the SDE does not directly act on the PDE.
For all realizations $\omega \in \Omega$, the control law $V(t,\omega)$ is in $L^2([0,T] , \mathbb{R})$, and therefore $(t,x) \mapsto -\rho(x) V(t,\omega)$ is in $L^2([0,T] \times [0,1] , \mathbb{R})$. 
Using~\cite[Thm 10.11]{brezis_functional_2011} and the fact that $z(0,\cdot) \in H^2$, 
we obtain the existence, uniqueness, and the above stated regularity of $z$. 
Since $z(t,\cdot) \in H^1$ , we have by Morrey's inequality \cite[Thm 9.12]{brezis_functional_2011} that
$$
\mathbb{E}\left[ \int^T_0 \vert z(t,0) \vert \; \mathrm{d}t \right] \leq C \mathbb{E}\left[ \int^T_0 \Vert z(t,\cdot) \Vert_{H^1} \; \mathrm{d}t \right] < \infty ,
$$
for some constant $C>0$, implying that $t \mapsto z(t,0) \in L^2_\mathcal{F}$,
and therefore $t \mapsto z(t,0) - \rho(0) U(t) \in L^2_\mathcal{F}$. We can now use \cite[Chap. 1 - Thm 6.3]{yong_stochastic_1999} to obtain the existence, uniqueness, and the above stated regularity regularity properties for the SDE solution $X$.
\end{proof}

As stated in the previous proof, $z \mapsto z(0)$ is a continuous linear form on $H^1([0,1])$. By Riesz's theorem \cite[Thm. 4.11]{brezis_functional_2011}, there exists a unique $\gamma_0 \in H^1([0,1])$ such that
$\forall z \in H^1([0,1]), z(0) = <\gamma_0 , z >_{H^1} $. 
In what follows, we use the notation $z(0) = \gamma_0^* z$ for the sake of conciseness. 

\subsection{Existence of an optimal controller}

By using the SPDE framework, we can now rewrite our optimal control problem as a Linear-Quadratic (LQ) problem in a Hilbert space. Let $\mathcal{H}$ be the Hilbert space $\mathcal{H} \triangleq \mathbb{R}^d \times \mathbb{R} \times H^1(0,1)$. We consider the augmented state $Z_t \triangleq \left( \begin{array}{c} X_t \\ Y_t \\ z_t \end{array} \right) \in \mathcal{H}$, where $Y_t$ stands for the signal $U(t)$. Therefore, we infer the following linear SPDE
\begin{equation}\label{eq:formal_SPDE_equivalent_linear_system_LQ}
\left\{ \begin{array}{l}
dZ_t = \bigl[ ( \Tilde{\Delta} + \Tilde{A} ) Z_t + \Tilde{B} V(t) \bigr]dt + \Tilde{C} Z_t dW_t. \\
Z_0  =  \left( \begin{array}{c} X_0 \\ U_0 \\ u_0 - \rho U_0 \end{array} \right) ,
\end{array} \right.
\end{equation}
with
\begin{align*}
& \Tilde{\Delta} \triangleq \left( \begin{array}{ccc} 0 & 0 & 0 \\ 0 & 0 & 0 \\ 0 & 0 & \Delta \end{array} \right), \quad \Tilde{A} \triangleq \left( \begin{array}{ccc} A & -B\rho(0) & B \gamma_0^* \\ 0 & \mu & 0 \\ 0 & 0 & c \end{array} \right), \\
& \Tilde{B} \triangleq  \left( \begin{array}{c} 0 \\ 1 \\ - \rho \end{array} \right)  , \quad \Tilde{C} \triangleq \left( \begin{array}{ccc} C & -D\rho(0) & D \gamma_0^* \\ 0 & 0 & 0 \\ 0 & 0 & 0 \end{array} \right)
\end{align*}

The cost \eqref{eq:original_LQ_cost_constant} can in turn be expressed as
\small
\begin{equation}\label{eq:original_LQ_cost_constant_SPDE_notation}
J(V,U_0) \triangleq \espE \left[ \int_0^T < \Tilde{Q} Z_t , Z_t>_{\mathcal{H}} dt + < \Tilde{G} Z_t , Z_t>_{\mathcal{H}} \right],
\end{equation}
\normalsize
where $\Tilde{Q}$ and $\Tilde{G}$ are positive self-adjoint operators defined by
$$
\Tilde{Q} \triangleq \left( \begin{array}{ccc} Q & 0 & 0 \\ 0 & r & 0 \\ 0 & 0 & 0 \end{array} \right) , \quad \Tilde{G} \triangleq \left( \begin{array}{ccc} G & 0 & 0 \\ 0 & 0 & 0 \\ 0 & 0 & 0 \end{array} \right).
$$
To guarantee that this LQ problem admits a unique solution, we need to add a regularizing term in the cost function, ensuring strong convexity with respect to the new control variable $V$. We therefore recast the modified cost function $J_\delta$ as
\begin{equation}\label{eq:new_LQ_cost_constant_SPDE_notation_regularized}
J_{\delta}(V,U_0) \triangleq J(V,U_0) + \delta \espE \left[ \int_0^T V(t)^2 dt\right],
\end{equation}
where $\delta > 0$ is a parameter that can be chosen small at will.
In this setting, our control problem is well-posed, and the optimal controller ensures a low-variance state, leading to reliable system regulation. We therefore focus now on finding the solution to
$$
\overline{J}_\delta \triangleq \min_{U_0, V} J_\delta(V, U_0).
$$
We can now use the results from \cite{lu_concise_2021} to obtain the existence and explicit form of the optimal feedback LQ controller.

\begin{proposition}\label{prop:existence_optimal_LQ_control}
There exists a unique optimal controller $\overline{V}$ that minimizes the cost function \eqref{eq:new_LQ_cost_constant_SPDE_notation_regularized} under the dynamic \eqref{eq:formal_SPDE_equivalent_linear_system_LQ}. It is in the form of a feedback law:
\begin{equation}\label{eq:formula_explicit_optimalcontrol_LQ}
\overline{V}(t) = - \delta^{-1} \Tilde{B}^* \Pi(t) Z_t,
\end{equation}
where $\Pi(\cdot)$ is a bounded self-adjoint positive operator in $\mathcal{H}$, defined as the unique mild solution to the operator-valued Riccati equation
\begin{equation}\label{eq:operator_riccati_optimal_LQ}
\left\{ \begin{array}{l}
\Dot{\Pi} + \Pi \bigl( \Tilde{\Delta} +  \Tilde{A} \bigr)  + \bigl( \Tilde{\Delta} +  \Tilde{A}^* \bigr) \Pi \\
\hspace{3em} + \Tilde{C}^* \Pi \Tilde{C} + \Tilde{Q} - \delta^{-1} \Pi \Tilde{B} \Tilde{B}^* \Pi = 0, \\
\Pi(T) = \Tilde{G}.
\end{array} \right.
\end{equation}
Moreover, the minimal cost is given by
\begin{equation}\label{eq:minimal_LQ_cost}
J(\overline{V}, U_0) = < \Pi(0) Z_0 , Z_0>_{\mathcal{H}}.
\end{equation}
\end{proposition}
\begin{proof}
    As a classic result from functional analysis,t the Laplacian $\Delta$ with domain $\text{Dom}(\Delta) = \{ z \in H^2, z_x(0) = z_x(1) = 0 \} $ generates a strongly continuous semigroup of operators on $H^1$ \cite{brezis_functional_2011}. Since the operators $\Tilde{A}$, $\Tilde{B}$ and $\Tilde{C}$ are bounded, we can apply theorem \cite[Thm 9.1]{lu_concise_2021} to obtain the desired result.
\end{proof}
We can further minimize the optimal cost by choosing an optimal value for the problem parameter $U_0$. For conciseness, we define
\begin{equation}\label{eq:definition_of_L_and_M_conciness_optimal_cost}
L_\rho = \left( \begin{array}{c} 0 \\ 1 \\ - \rho \end{array} \right), \quad M_0 = \left( \begin{array}{c} X_0 \\ 0 \\ u_0 \end{array} \right) ,
\end{equation}
\begin{proposition}\label{prop:minimal_value_U_0_LQ}
The cost function $J_\delta(V,Z_0)$ is minimized for the parameters $V = \overline{V}$ as given in \eqref{eq:formula_explicit_optimalcontrol_LQ}, and 
\begin{equation}\label{eq:optimal_initial_condition_control_U_0}
\overline{U}_0 = - \frac{< \Pi(0) L_\rho , M_0 >_{\mathcal{H}}}{< \Pi(0) L_\rho , L_\rho >_{\mathcal{H}}}
\end{equation}
if $< \Pi(0) L_\rho , L_\rho >_{\mathcal{H}} \neq 0$, or $U_0 = 0$ otherwise.
\end{proposition}
\begin{proof}
Remark that $L_\rho$ and $M_0$ are defined such that $Z_0 = L_\rho U_0 + M_0$. Let us define $F(U_0) \triangleq J_\delta(\overline{V}, L_\rho U_0 + M_0)$ the functional we seek to minimize. We have through Proposition \ref{prop:existence_optimal_LQ_control} that 
$$
F(U_0) = < \Pi(0) (L_\rho U_0 + M_0) , (L_\rho U_0 + M_0)>_{\mathcal{H}},
$$
which is a quadratic polynomial in $U_0$. It is straightforward to show that, if $< \Pi(0) L_\rho , L_\rho >_{\mathcal{H}} \neq 0$, then $F(U_0)$ is minimal for $\overline{U}_0$ as defined in \eqref{eq:optimal_initial_condition_control_U_0}. If $< \Pi(0) L_\rho , L_\rho >_{\mathcal{H}} \neq 0$, then $F(U_0)$ is constant and we can choose $\overline{U}_0 = 0$.
\end{proof}

Numerically solving the infinite-dimensional Riccati equation requires discretization \cite{banks_approximation_1994}, and choosing an appropriate method is crucial for ensuring both convergence and feasibility. In the deterministic setting, several numerical approaches have been proposed, including spectral methods that discretize the state by projecting onto the first eigenvectors of the Laplacian operator \cite{banks_approximation_1994, ito_strong_1987, morris_controller_2020}, as well as more general Galerkin methods \cite{kroller_convergence_1991} based on finite element techniques. The Galerkin approach has also been extended to the stochastic setting in \cite{prohl_strong_2022}, and a splitting method was introduced in \cite{damm_numerical_2017} to approximate the Riccati solution.
In the next section, we extend the spectral method to the finite-time LQ control problem for SPDEs, to the best of our knowledge, for the first time.


\section{Discretization of the System}\label{DS}

To compute an easy-to-implement numerical approximation of the optimal control \eqref{eq:formula_explicit_optimalcontrol_LQ}, we employ a spectral approach~\cite{morris_controller_2020}.
We start to discretize the system dynamics onto a finite-dimensional subspace $\mathcal{H}_N \subset \mathcal{H}$. We then solve the optimal control problem within $\mathcal{H}_N$ and demonstrate that the resulting finite-dimensional optimal controller provides a good approximation of the true optimal controller. Furthermore, we establish the convergence of both the approximated controller and its associated cost to the exact optimal controller and its corresponding minimal cost, respectively.

\subsection{Finite-Dimensional System}

We consider the Laplacian eigenfunctions $\phi_n$ and eigenvalues $\lambda_n$,  
verifying
\begin{equation}\label{eq:stationary_PDE_eigenvector_laplacian}
\left\{ \begin{array}{l} - \Delta \phi_n(x)  = \lambda_n \phi_n \\ 
\left(\phi_n\right)_x(0) = 0, \quad \left(\phi_n\right)_x(1) = 0 ,
\end{array}  \right.
\end{equation}
and we arrange these eigenvalues in increasing order, i.e., $0 = \lambda_0 < \lambda_1 < \dots$.
In our case, we can explicitly solve the Sturm-Liouville problem, 
the solutions being given by
\begin{equation}\label{eq:solutions_eigenvalue_eigenvector_problem}
\lambda_n = \pi^2 n^2 , \quad \phi_n = A_n cos(n \pi x) ,
\end{equation}
with coefficients $A_n = \sqrt{ \frac{2}{1 + \pi^2 n^2} }$ chosen such that $\phi_n$ is a orthonormal basis in $H^1$. 
In what follows, we denote the discretization of $H^1$ with $H^1_N \triangleq \text{Span}(\phi_0, \dots, \phi_N) $, and we accordingly denote by $\mathcal{H}_N \triangleq \mathbb{R}^d \times \mathbb{R} \times H^1_N$ the discretization of $\mathcal{H}$.
We denote by $P_N$ the projection of $\mathcal{H}$ onto $\mathcal{H}_N$, that is given by projecting onto the first $N$ eigenvectors $\phi_0, \dots, \phi_N$. The explicit expression of $P_N$ is
\begin{equation}\label{eq:explicit_expression_projector_P_N}
    P_N \left( \begin{array}{c} X \\  Y \\ z \end{array} \right) = \left( \begin{array}{c} X \\ Y \\  \sum_{n=0}^N < z , \phi_n >_{H^1} \phi_n \end{array} \right)
\end{equation}
This projector is orthogonal, as for $Z \in \mathcal{H}$ and $\phi \in \mathcal{H}_N$
$$  
 < P_N Z , \phi >_{\mathcal{H}_N} = < P_N Z , \phi >_{\mathcal{H}} = < Z , \phi >_{\mathcal{H}}.
$$
The adjoint operator $P_N^*$ is therefore the injection of $\mathcal{H}_N$ back into $\mathcal{H}$. 
We now solve the optimal control problem for the approximation of \eqref{eq:formal_SPDE_equivalent_linear_system_LQ} on $\mathcal{H}_N$. We define the matrices $\Tilde{\Delta}_N = P_N \Tilde{\Delta} P_N^* $, $\Tilde{A}_N = P_N \Tilde{A} P_N^* $ $\Tilde{C}_N = P_N \Tilde{C} P_N^* $, and $\Tilde{B}_N = P_N \Tilde{B} $,
and consider the approximated state $Z_t^N \in \mathcal{H}_N$ that verifies the linear SDE
\begin{equation}\label{eq:formal_SPDE_discretized_linear_system_LQ_finite_dim}
\left \{ \begin{array}{l}
dZ_t^N = \bigl[ ( \Tilde{\Delta}_N + \Tilde{A}_N ) Z_t^N + \Tilde{B}_N V(t) \bigr] dt + \Tilde{C}_N Z_t^N dW_t,\\
Z^N_0 =  P_N Z_0 .
\end{array} \right.
\end{equation}
The different matrices can be written explicitly along the basis of vectors $(\phi_0, \dots, \phi_N)$ as
\small
\begin{align*}
\Tilde{\Delta}_N & = \left( \begin{array}{ccc} 0 & 0 & 0 \\ 0 & 0 & 0 \\ 0 & 0 & - L_N \end{array} \right),~ \Tilde{A}_N = \left( \begin{array}{ccc} A & -B\rho(0) & B \left(\gamma_0^N\right)^T \\ 0 & \mu & 0 \\ 0 & 0 & c I_N \end{array} \right) \\
\Tilde{C}_N &  = \left( \begin{array}{ccc} C & -D\rho(0) & D \left(\gamma_0^N\right)^T \\ 0 & 0 & 0 \\ 0 & 0 & 0 \end{array} \right), \quad \Tilde{B}_N = \left( \begin{array}{c} 0 \\ 1 \\ - \rho^N \end{array} \right) 
\end{align*}
\normalsize
with $L_N \triangleq \text{diag}( \lambda_0, \dots \lambda_N) $, $\gamma_0^N \triangleq \bigl( \phi_0(0) , \dots, \phi_N(0) \bigr)^T$ and $\rho^N \triangleq \bigl( < \rho , \phi_0 >_{H^1},~ \dots,~ < \rho , \phi_N >_{H^1} \bigr)^T$.

\subsection{Approximation of the optimal feedback control}
We now define the approximated cost $J^N_\delta(V,Z_0^N)$ as
\begin{equation}\label{eq:new_LQ_cost_constant_SPDE_notation_regularized_approx_finite_dim}
\begin{split}
J^N_{\delta}(V,Z_0) & \triangleq   \espE \left[ \int_0^T \bigl( (Z_t^N)^T \Tilde{Q}_N Z_t^N + \delta V(t)^2 \bigr) dt\right] \\
& \quad + \espE \left[ (Z_T^N)^T \Tilde{G}_N Z_T^N \right],
\end{split}
\end{equation}
where $\Tilde{Q}_N \triangleq P_N \Tilde{Q} P_N^* $ and $\Tilde{G}_N \triangleq P_N \Tilde{G} P_N^* $.
Classic results from optimal linear quadratic control \cite{yong_stochastic_1999} yields the following optimal controller.
\begin{proposition}\label{prop:existence_optimal_LQ_control_finite_dim}
There exists a unique optimal controller $\overline{V}_N$ that minimizes the cost \eqref{eq:new_LQ_cost_constant_SPDE_notation_regularized_approx_finite_dim} under the dynamic \eqref{eq:formal_SPDE_discretized_linear_system_LQ_finite_dim}. Moreover, this optimal controller is given by the feedback law
\begin{equation}\label{eq:formula_explicit_optimalcontrol_LQ_finite_dim}
\overline{V}_N(t) = - \delta^{-1} \Tilde{B}_N^T \Pi_N(t) Z^N_t,
\end{equation}
where $\Pi_N(\cdot)$ is a positive symmetric matrix in $\mathcal{L}(\mathcal{H}_N)$, and is the unique solution to the matrix-valued Riccati equation
\begin{equation}\label{eq:operator_riccati_optimal_LQ_finite_dim}
\left\{ \begin{array}{l}
\Dot{\Pi}_N + \Pi_N \bigl( \Tilde{\Delta}_N +  \Tilde{A}_N \bigr)  + \bigl( \Tilde{\Delta}_N +  \Tilde{A}_N^* \bigr) \Pi_N \\
\hspace{3em} + \Tilde{C}_N^* \Pi_N \Tilde{C}_N + \Tilde{Q}_N - \delta^{-1} \Pi_N \Tilde{B}_N \Tilde{B}_N^* \Pi_N = 0, \\
\Pi_N(T) = \Tilde{G}_N.
\end{array} \right.
\end{equation}
Finally, the associated minimal cost is given by
\begin{equation}\label{eq:minimal_LQ_cost_finite_dim}
J(\overline{V}_N, Z^N_0) = < \Pi_N(0) Z^N_0 , Z^N_0>_{\mathcal{H_N}}.
\end{equation}
\end{proposition}
As done in Proposition \ref{prop:minimal_value_U_0_LQ}, one easily shows that there exists an optimal value of $U_0$, denoted $\overline{U}_0^N$, that further minimizes the cost $J(\overline{V}_N, Z^N_0)$. The feedback controller $\overline{V}_N$ is straightforward to compute in practice \cite{yong_stochastic_1999}. We show in the next section that $\overline{V}_N$ is a good approximation of the optimal control $\overline{V}$ when $N$ is big enough. 


\section{Convergence Results}\label{CVR}

\subsection{Convergence of the Dynamic Operators}
In what follows, we denote by $S(\cdot)$ and $S_N(\cdot)$ the strongly continuous semi-group generated respectively by $\Tilde{\Delta}$ and $\Tilde{\Delta}_N$.
It is a known result that the spectral approximation of the Laplacian $\Tilde{\Delta}_N$ 
verifies the Trotter-Kato theorem \cite[Thm 4.3]{morris_controller_2020}. Consequently, we have that 
\begin{equation}\label{eq:Trotter_Kato_consequence_convergence_S}
\forall Z \in \mathcal{H},~ \lim_{n \rightarrow + \infty} \sup_{0 \leq t \leq T} \Vert P_N^* S_N(t) P_N Z - S(t)Z \Vert_\mathcal{H} = 0,
\end{equation}
Moreover, we can verify that all the other bounded operators converge strongly pointwise, that is
\begin{equation}\label{eq:convergence_bounded_operators}
\begin{split}
\forall Z \in \mathcal{H}, \quad & \Vert \Tilde{A} Z - P_N^* \Tilde{A}_N P_N Z \Vert_\mathcal{H} \rightarrow 0,  \\
& \Vert \Tilde{C} Z - P_N^* \Tilde{C}_N P_N Z \Vert_\mathcal{H} \rightarrow 0, \\
&  \Vert \Tilde{B} - P_N^* \Tilde{B}_N \Vert_\mathcal{H} \rightarrow 0.
\end{split}
\end{equation}
These convergences also hold for the adjoints of the operators $A^*$, $B^*$ and $C^*$.  

\subsection{Convergence of the Feedback Operator}
In this subsection we prove the strong convergence of the feedback operator $\Pi_N$ towards $\Pi$.
\begin{proposition}\label{prop:convergence_of_the_Riccati_operator}
    Let $Z \in \mathcal{H}$. We have that
\begin{equation}\label{eq:convergence_Riccati_operator}
\begin{split}
& \quad \Vert \Pi(t) Z - P_N^* \Pi_N(t) P_N Z \Vert_\mathcal{H} \rightarrow 0, \\
& \sup_{t \in [0,T]} \Vert \Pi(t) Z - P_N^* \Pi_N(t) P_N Z \Vert_\mathcal{H} \rightarrow 0.
\end{split} 
\end{equation}
\end{proposition}
\begin{proof}
The Riccati equation \eqref{eq:operator_riccati_optimal_LQ} closely resembles the classical Riccati equation solved in the deterministic setting \cite{morris_controller_2020}, with the main difference being the additional term  $C^* \Pi C $ which is not particularly difficult to handle. However, since we could not find a proof that allows for a direct adaptation, we provide a full proof here using a standard approach.
Let $t \in [0,T]$ and $Z \in \mathcal{H}$, we prove this convergence result through an estimate on the difference $\Pi(t) Z - P_N^* \Pi_N(t) P_N Z$. Let $\tau \in [t,T]$, we define by $F(\tau)Z$ and $F_N(\tau)Z$ the $C^1$ functionals 
\begin{align*}
F(\tau)Z & \triangleq S(\tau - t)^* \Pi(\tau) S(\tau - t)Z, \\
F_N(\tau)Z & \triangleq P_N^* S_N(\tau - t)^* \Pi_N(\tau) S_N(\tau - t) P_N Z. 
\end{align*}
Through the fundamental theorem of analysis, we have
\begin{align*}
& F(T)Z - F(t)Z  = \int_t^T \Dot{F}(\tau) Z d\tau  \\
& \quad = \int_t^T S(\tau - t)^* \bigl( \Tilde{\Delta}^* \Pi(\tau) + \Pi(\tau) \Tilde{\Delta} \bigr) S(\tau - t)Z d\tau \\
& \quad  +    \int_t^T S(\tau - t)^*  \Dot{\Pi}(\tau)  S(\tau - t)Z d\tau.
\end{align*}
Using the Riccati equation \eqref{eq:operator_riccati_optimal_LQ}, and the fact that $F(t) = \Pi(t)$, we obtain
\begin{align*}
& \Pi(t) Z  = S(T - t)^* \Tilde{G} S(T - t)Z \\
& \quad + \int_t^T S(\tau - t)^* \biggl( \Tilde{A}^* \Pi(\tau) + \Pi(\tau) \Tilde{A} \\
& \hspace{5em} + \Tilde{Q} + \Tilde{C}^* \Pi(\tau) \Tilde{C} \\
& \hspace{5em} - \delta^{-1}  \Pi(\tau) \Tilde{B} \Tilde{B}^* \Pi(\tau) \biggr) S(\tau - t)Z d\tau 
\end{align*}
Similarly, we have 
\begin{align*}
& P_N^* \Pi_N(t) P_N Z  = P_N^* S_N(T - t)^* \Tilde{G}_N S_N(T - t) P_N Z \\
& + \int_t^T S_N(\tau - t)^* \biggl( \Tilde{A}_N^* \Pi_N(\tau) + \Pi_N(\tau) \Tilde{A}_N \\
& \hspace{3em} + \Tilde{Q}_N + \Tilde{C}_N^* \Pi_N(\tau) \Tilde{C}_N \\
& \hspace{3em} - \delta^{-1}  \Pi_N(\tau) \Tilde{B}_N \Tilde{B}_N^* \Pi_N(\tau) \biggr) S_N(\tau - t) P_N Z d\tau.
\end{align*}
All operators $S_N(\cdot), \Pi_N(\cdot), \Tilde{A}_N, \Tilde{B}_N, \Tilde{C}_N, \Tilde{G}_N, \Tilde{Q}_N, P_N$ have an operator norm that is uniformly bounded (for $\Pi_N$, see the method proposed in \cite{kroller_convergence_1991} for a proof). Therefore, there exists a (possibly overloaded) positive constant $K$, depending on the bounds of the operator norms, such that
\small
\begin{align*}
& \Vert \Pi(t) Z - P_N^* \Pi_N(t) P_N Z \Vert_\mathcal{H} \\
& \leq K \int_t^T \Vert \Pi(\tau) Z - P_N^* \Pi_N(\tau) P_N Z \Vert_\mathcal{H} d\tau \\
& \quad + K \Bigl( \Vert \Tilde{A}Z  - P_N^* \Tilde{A}_N P_N Z\Vert_\mathcal{H} + \Vert \Tilde{C}Z  - P_N^* \Tilde{C}_N P_N Z\Vert_\mathcal{H} \\
& \hspace{3em} + \Vert \Tilde{Q}Z  - P_N^* \Tilde{Q}_N P_N Z\Vert_\mathcal{H} + \Vert \Tilde{G}Z  - P_N^* \Tilde{G}_N P_N Z\Vert_\mathcal{H} \\
& \hspace{3em} + \Vert \Tilde{B}\Tilde{B}^* Z  - P_N^* \Tilde{B}_N \Tilde{B}_N^* P_N Z \Vert_\mathcal{H} \\
& \hspace{3em} + \sup_{\tau \in [0,T]} \Vert S(\tau) Z  - P_N^* S_N(\tau) P_N Z\Vert_\mathcal{H} \Bigr).
\end{align*}
\normalsize
We conclude the proof thanks to Gronwall's Lemma, and \eqref{eq:Trotter_Kato_consequence_convergence_S} and \eqref{eq:convergence_bounded_operators}.
\end{proof}

We can now show the convergence of the controller and the optimal cost.

\subsection{Convergence of the Controller and the Optimal Cost}

We show first the strong convergence of the closed-loop state, then of the controller and finally of the optimal cost.

\begin{proposition}\label{prop:convergence_closed_loop_state}
Let $\overline{Z}_t$ be the solution to the closed loop dynamic \eqref{eq:formal_SPDE_equivalent_linear_system_LQ} with the controller $\overline{V}$ defined in \eqref{eq:formula_explicit_optimalcontrol_LQ}, and let $\overline{Z}^N_t$ be the solution of \eqref{eq:formal_SPDE_discretized_linear_system_LQ_finite_dim} with $\overline{V}_N$ defined in \eqref{eq:formula_explicit_optimalcontrol_LQ_finite_dim}. Then
\begin{equation}\label{eq:convergence_closed_loop_state_C2}
    \Bigl \Vert \overline{Z}_t - P_N^* \overline{Z}_t^N \Bigr \Vert_{C^2_\mathcal{F}([0,T]; \mathcal{H})} \rightarrow 0.
\end{equation}
\end{proposition}
\begin{proof}
Similarly to  what has been done in the proof of Proposition \ref{prop:convergence_of_the_Riccati_operator}, we first express $\overline{Z}_t$ and $\overline{Z}_t^N $ as integrals involving uniformly bounded operators. We then bound their differences and establish the desired result using the B-D-G inequality 
for the stochastic integral  given in equation  \eqref{eq:Burkhloder_Davis_Gundy_ineq_stoch_int}, combined with Gronwall’s Lemma.
We have the following expression for $\overline{Z}_t$,
\begin{align*}
\overline{Z}_t & = S(t)Z_0 +  \int_0^t S(t-s) \Tilde{C} \overline{Z}_t dW_t  \\
& + \int_0^t S(t-s) \bigl( \Tilde{A} - \delta^{-1} \Tilde{B} \Tilde{B}^* \Pi(t) \bigr) \overline{Z}_t dt.
\end{align*}
Similarly,
\begin{align*}
\overline{Z}^N_t & = S_N(t)Z^N_0 +  \int_0^t S_N(t-s) \Tilde{C}_N \overline{Z}^N_t dW_t  \\
& + \int_0^t S_N(t-s) \bigl( \Tilde{A}_N - \delta^{-1} \Tilde{B}_N \Tilde{B}_N^* \Pi_N(t) \bigr) \overline{Z}^N_t dt.
\end{align*}
Let us focus on the stochastic integral. Due to the B-D-G inequality \eqref{eq:Burkhloder_Davis_Gundy_ineq_stoch_int}, there exists $K>0$ such that
\footnotesize
\begin{align*}
& \espE \left[  \sup_{0 \leq s \leq t}  \left \Vert \int_0^s \Bigl( S(s-\tau) \Tilde{C} \overline{Z}_\tau - P_N^*  S_N(s-\tau) \Tilde{C}_N \overline{Z}^N_\tau \Bigr) dW_\tau     \right \Vert^2_\mathcal{H}  \right] \\
& \leq K \int_0^t \espE \left[ \sup_{0 \leq \tau \leq s}  \left \Vert S(s-\tau) \Tilde{C} \overline{Z}_\tau - P_N^*  S_N(s-\tau) \Tilde{C}_N \overline{Z}^N_\tau \right \Vert^2_\mathcal{H}  \right] ds  \\
& \leq K \int_0^t \espE \left[ \sup_{0 \leq \tau \leq s}  \left \Vert \overline{Z}_\tau -  \overline{Z}^N_\tau \right \Vert^2_\mathcal{H}  \right] ds \\ 
& + K \int_0^t \espE \left[ \sup_{0 \leq \tau \leq s}  \left \Vert \bigl( \Tilde{C} - P_N^* \Tilde{C}_N P_N \bigr) \overline{Z}_\tau \right \Vert^2_\mathcal{H}  \right] ds \\
& + K \int_0^t \espE \left[ \sup_{0 \leq \tau \leq s}  \left \Vert \bigl( S(s - \tau) - P_N^* S_N(s - \tau) P_N \bigr) \overline{Z}_\tau \right \Vert^2_\mathcal{H}  \right] ds.
\end{align*}

\normalsize
Note that we can add $P_N P_N^* = I_N$ anywhere in the expression of $\overline{Z}^N_t$ to factorize each term appropriately. To prove the convergence to $0$ of the last two integrals, we use the following lemma from \cite{gibson_riccati_1979}. 
\vspace{0.6em}
\begin{lemma}\label{lem:gibson_convergence_operator_times_function}
\cite[Lem 5.1]{gibson_riccati_1979} Let $\Psi(\cdot) : [0,T] \rightarrow \mathcal{L}(\mathcal{H})$ and for $N \geq 1$, $\Psi_N(\cdot) : [0,T] \rightarrow \mathcal{L}(\mathcal{H})$ such that the operator norm of $\Psi_N(t)$ is uniformly bounded for $N$ and $t$, and $\forall Z \in \mathcal{H} $, $\Psi_N(t)Z \rightarrow \Psi(t)Z$ uniformly in $t$. Let $g : [0,T] \rightarrow \mathcal{H}$ be a continuous function, then
$$\sup_{0\leq t \leq T} \bigl \Vert \Psi_N(t)g(t) - \Psi(t)g(t) \bigr \Vert_\mathcal{H} \rightarrow 0.
$$
\end{lemma}
Since for all realizations $\omega \in \Omega$, the function $t \mapsto \overline{Z}_t(\omega)$ is continuous in $\mathcal{H}$, we can apply the previous lemma pointwise in $\omega$. We then use the boundedness of the operators and of $\overline{Z}_t$ in $C^2_\mathcal{F}([0,T];\mathcal{H})$ to apply the dominated convergence on the integral of the expectation and show convergence to 0.  
Using this method, we can show that
\begin{align*}
& \espE \left[ \sup_{0 \leq s \leq t}  \left \Vert \overline{Z}_s -  \overline{Z}^N_s \right \Vert^2_\mathcal{H}  \right] \leq  \Theta_N \\
& \qquad  +  K \int_0^t \espE \left[ \sup_{0 \leq \tau \leq s}  \left \Vert \overline{Z}_\tau -  \overline{Z}^N_\tau \right \Vert^2_\mathcal{H}  \right] ds
\end{align*}
with $\Theta_N \rightarrow 0$. We can now conclude using Gronwall's Lemma.
\end{proof}
We can now show convergence of the control and of the cost.
\begin{proposition}
The following holds as $N \rightarrow \infty$:
\begin{equation}\label{eq:convergence_optimal_control_and_cost}
\begin{array}{c}
      \Vert \overline{V} - \overline{V}_N \Vert_{C^2_\mathcal{F}([0,T], \mathbb{R})} \rightarrow 0, \\
      \overline{U}_0^N \rightarrow \overline{U}_0 , \\
      J^N_\delta \bigl( \overline{V}_N , \overline{U}_0^N ) \rightarrow  J_\delta \bigl( \overline{V} , \overline{U}_0 ).
\end{array}
\end{equation}
\end{proposition}
\begin{proof}
Proposition \ref{prop:convergence_closed_loop_state} combined with Lemma \ref{lem:gibson_convergence_operator_times_function} yields directly the results.
\end{proof}
These convergence results confirm that the easily implementable controls \( \overline{V}_N \) serve as accurate approximations of \( \overline{V} \), as they become sufficiently close for large \( N \).


\section{Simulations}\label{SIM}

To show the effectiveness of our control method, we consider an academic example within the framework of \eqref{eq:systeme_original_coupled}, with a real valued SDE. In our simulations, we chose for the SDE the parameters $A = 2, B = 2, C = 1, D = 0.5$ and $X_0 = 1$. The PDE parameters are $c = 0.5$ and $u_0 = 1$. Note that with this choice of parameters, both the PDE and SDE are unstable. We apply the control strategy defined in proposition \ref{prop:existence_optimal_LQ_control_finite_dim} with $N=3$, $Q = 10, R = 1, G = 10$ and $ \delta = 0.5$.

\begin{figure}[ht!]
\centering
  \includegraphics[width=0.80\linewidth]{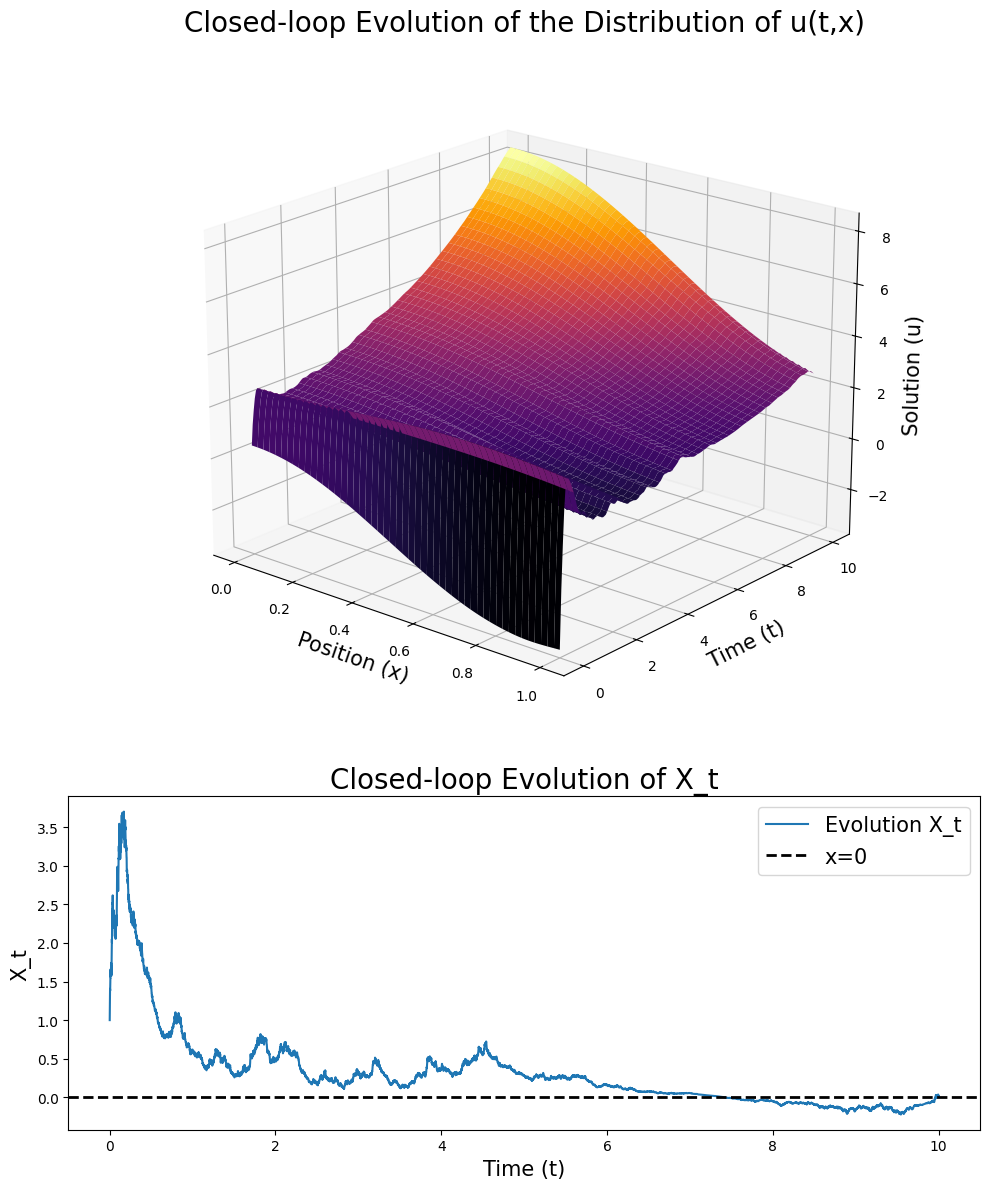}
  \caption{\centering Realization of a closed-loop trajectory with $N=3$}
  \label{fig:plot_coupled_sys_closed}
\end{figure}

Figure \ref{fig:plot_coupled_sys_closed} shows that despite an initial spike induced by the PDE, the controller effectively stabilizes the state, overcoming instabilities and disturbances. However, the PDE state itself remains unstabilized, which is expected since it is not penalized in the cost function. Instead, it naturally adjusts to minimize the variance of the SDE state.


\section{Conclusion and Perspectives}\label{CCL}

In this paper, we introduced a method for the optimal control of an interconnected system composed of a heat PDE cascaded into a linear SDE system. Our approach extends spectral approximation methods developed for parabolic PDE to the stochastic setting. We established the convergence of the approximated optimal controller to the true optimal control of the original system, as well as the convergence of the corresponding closed-loop state dynamics.
Future research directions include extending our framework to more general PDE operators with sufficiently fast-converging finite-dimensional approximations. Additionally, we aim to investigate the case of stochastic systems with possibly nondeterministic coefficients in the SDE, further enhancing the applicability of our results.


\bibliographystyle{ieeetr}
\bibliography{biblio_these_GV_Modif}

\end{document}